\documentclass[12pt,reqno]{amsart}
\usepackage{hyperref}
\usepackage{color, bm, amscd, tikz-cd}
\newcommand{\bydef}{\stackrel{\rm def}{=}}

\newcommand{\cB}{{\mathcal B}}

\newcommand{\cD}{{\mathcal D}}
\newcommand{\cE}{{\mathcal E}}
\newcommand{\cF}{{\mathcal F}}

\newcommand{\cH}{{\mathcal H}}

\newcommand{\cK}{{\mathcal K}}

\newcommand{\cM}{{\mathcal M}}

\newcommand{\cP}{{\mathcal P}}

\newcommand{\bT}{{\mathbb{T}}}

\newcommand{\la}{\langle}
\newcommand{\ra}{\rangle}
\newcommand{\bD}{{\mathbb D}}

\newtheorem{thm}{Theorem}[section]

\newtheorem{corollary}[thm]{Corollary}
\newtheorem{lemma}[thm]{Lemma}

\theoremstyle{definition}

\newtheorem{definition}[thm]{Definition}

\newtheorem{example}[thm]{Example}

\numberwithin{equation}{section}

\author{Abhay Jindal}
\address{Department of Mathematics\\
	Indian Institute of Science\\
	Bangalore 560012, India}
\email{abjayj@iisc.ac.in}
\author{Poornendu Kumar}
\address{Department of Mathematics, University of Manitoba, Winnipeg, Manitoba, Canada, R3T 2N2}
\email{poornendu.kumar@umanitoba.ca}
\usepackage{fancyhdr}
\begin{document}
	\thanks{
		{\em Keywords and phrase}: Beurling-Lax-Halmos Theorem, Isometries, Pentablock, Rational Dilation, Wold Decomposition, }
	\title{Operator theory on the pentablock}
	\maketitle
	\begin{abstract}
		The pentablock, denoted as $\cP,$ is defined as follows:  
		$$\cP= \left\{ (a_{21}, {\rm tr}(A), {\rm det}(A)) :  A = [a_{ij}]_{2 \times 2} \text{ with } \|A\|<1 \right\}.$$
		It originated from the work of Agler--Lykova--Young in connection with a particular case of the $\mu$-synthesis problem. It is a non-convex, polynomially convex, $\mathbb{C}$-convex, star-like  about the origin, and inhomogeneous domain. 
		
		This paper deals with operator theory on the pentablock. We study pentablock unitaries and isometries, providing an algebraic characterization of pentablock isometries. En route, we provide the Wold-type decomposition for pentablock isometries, which consists of three parts: the unitary part, the pure part, and a new component. We define this novel component as the quasi-pentablock unitary and provide a functional model for it. Additionally, a model for a class of pure pentablock isometries has been found, along with some examples.  Furthermore, a representation resembling the Beurling-Lax-Halmos paradigm has been presented for the invariant subspaces of pentablock pure isometries.
	\end{abstract}
	
	\section{Introduction and Motivation }
	Consider a bounded domain $\Omega $ of the $d$-dimensional complex Euclidean space $\mathbb{C}^d$ and a tuple $\underline{T}= (T_1, \dots T_d)$ of commuting bounded operators on a complex separable Hilbert space $\cH$ such that
	\begin{enumerate}
		\item the Taylor joint spectrum $\sigma(\underline{T})$ is contained in  $\overline{\Omega}$ (see \cite{Curto, Vas} for a discussion on the Taylor joint spectrum), and
		\item  \begin{equation}\label{VN}
			\left\|r(T_1, \dots, T_d) \right\|\leq{\|r\|_{\infty, \overline{\Omega}}}
		\end{equation} 
		for any rational function $r$ with poles outside $\overline{\Omega}$ where $\|.\|_{\infty, \overline{\Omega}}$ represents the sup norm on $\overline{\Omega}$. 
	\end{enumerate}
	Such an $\Omega$ is called a {\em spectral set} for $\underline{T}.$ If the inequality in \eqref{VN} holds for all matrix-valued rational functions as well, then $\Omega$ is called a {\em complete spectral set} for $\underline{T}$. 
	
	A $\partial\Omega$-normal dilation of the tuple $\underline{T}$ is a commuting tuple of normal operators $\underline{N} = (N_1, \dots, N_d)$ defined on a Hilbert space $\mathcal{K}$ which contains $\mathcal{H}$ as a closed subspace, with the following conditions:
	\begin{enumerate}
		\item $\sigma(\underline{N})$ is a subset of the distinguished boundary $b\Omega$ of $\Omega$, and 
		\item for every rational function $r$ with poles outside $\overline{\Omega}$ we have $$r(T_1, \dots, T_d) = P_{\cH}r(N_1, \dots, N_d)|_{\cH},$$ 
		where $P_{\cH}$ is the orthogonal projection  from $\cK$ onto $\cH$. 
	\end{enumerate} 
	A remarkable theorem of Arveson states that $\Omega$ is a complete spectral set for $\underline{T}$ if and only if $\underline{T}$ has a $\partial \Omega-$normal dilation; see \cite{Arveson}. 
	
	A tuple $\underline{T}$ is called a \textit{$\Omega-$contraction} if $\Omega$ is a spectral set for $\underline{T}.$ A commuting tuple of normal operators $\underline{N} = (N_{1},\dots, N_{d})$ is called a \textit{$\Omega-$unitary} if $\sigma(\underline{N})$ is a subset of the distinguished boundary $b\Omega$. An \textit{$\Omega-$isometry} is the restriction of a $\Omega-$unitary to a joint invariant subspace. We say that rational dilation holds for $\Omega$ if $\Omega$ is a complete spectral set for all $\Omega-$contractions. Else, rational dilation fails.
	
	Historically, operator theorists have been intrigued by whether rational dilation holds for subsets of the complex plane or higher-dimensional Euclidean spaces. Rational dilation is known to hold for simply connected planar domains \cite{von} and for doubly connected planar domains \cite{Agler-H}. However, if a planar domain has multiple holes, then rational dilation fails \cite{DM}. Moving on to multivariable domains, the unit polydisc $\mathbb{D}^d$ has been extensively studied. Rational dilation is established for $d=2$ by Ando \cite{Ando} (also see \cite{Ball-Sau3} for two new constructive proofs of Ando's Theorem along with other model-theoretic results). Still, it does not hold for $d>2$; see \cite{Par}.  A special type of algebraic varieties, known as distinguished varieties (see \cite{AM, BKS}), has been extensively studied from the perspective of operator theory. However, rational dilation does not hold for these varieties, as shown in \cite{DU, DJM}.
	
	Recently, there has been a surge of interest in studying the rational dilation problem within various multivariable domains. These domains were initially discovered due to their relevance to the $\mu$-synthesis problem in Robust Control Theory \cite{Doyle, Dull}. We wish to underscore particular importance of three specific domains: the symmetrized bidisc $\mathbb{G}$ (introduced in \cite{AY} by Agler--Young), the tetrablock $\mathbb{E}$ (introduced in \cite{AYW} by Abouhajar--Young--White), and the pentablock $\mathcal{P}$ (introduced in \cite{ALY-JMAA} by Agler--Lykova--Young).
	These domains exhibit several commonalities, including non-convexity and polynomial convexity, and are categorized as inhomogeneous domains. Additionally, it is worth noting that the Lempert Theorem holds on the symmetrized bidisc \cite{Lem-Co}, and the tetrablock \cite{Lem-Tetra}. However, its validity on the pentablock remains unknown. Approximation by specific types of rational inner functions is another important property that the symmetrized bidisc and the tetrablock enjoy; see \cite{BLMS}. It is not yet known for the pentablock. In terms of dilation, 
	Agler and Young \cite{AY} proved that the rational dilation holds for $\mathbb{G}$ (see also \cite{TPS-Adv}). The rational dilation problem for the tetrablock poses certain uncertainties. However, some progress has been made in this regard; see \cite{Ball-Sau, Ball-Sau2, Tirtha-Tetrablock, BB, Pal}. No domain in $\mathbb{C}^3$ is known where the rational dilation holds.
	
	In this article, the domain of our interest is the pentablock. The pentablock is nothing but the image of the norm unit ball of $2\times 2$ matrices under the following map  $$\pi: A = [a_{ij}]\mapsto (a_{21}, \operatorname{tr}(A), \operatorname{det}(A)),$$ see \cite{ALY-JMAA}. We denote the closure of $\mathcal P$ by $\overline{\mathcal P}$. The set $\cP \subset \mathbb{C}^{3}$ is non-convex, polynomially convex, and star-like about the origin, see \cite{ALY-JMAA}. The pentablock is an inhomogeneous domain; see \cite{Kos-Penta}. The pentablock $\cP$ is closely related to the symmetrized bidisc $\mathbb{G},$ which is defined by 
	$$\mathbb{G} = \{ ({\rm tr}(A), {\rm det}(A)): A = [a_{ij}]_{2 \times 2} \text{ with } \|A\| <1\}.$$ To be more precise, the pentablock can be viewed as a Hartogs domain in $\mathbb{C}^{3}$ over the symmetrized bidisc. The complex geometry and function theory of the pentablock were further developed in \cite{ALY-JMAA, AL-Penta, Su-Penta, PZ}. In this paper, we study the pentablock from the operator theoretical point of view. 
	
	Pentablock unitaries are the candidates for the dilation triple of pentablock contractions. Consequently, a thorough characterization of pentablock unitaries becomes imperative. Section \ref{PU} of this article presents various characterizations of a triple being a pentablock unitary. 
	
	A pentablock isometry $(V_{1},V_{2},V_{3})$ is called pure if $V_{3}$ is a pure isometry. Section \ref{WD} gives a Wold-type decomposition for pentablock isometries. Unlike the earlier decomposition known in different domains, which treated any $\Omega-$isometry as a direct sum of a $\Omega-$unitary and a pure $\Omega-$isometry, this decomposition will incorporate an additional term. We give a functional model for this additional term. Section \ref{P-Iso} gives an algebraic characterization of pentablock isometries. Pure isometries play a significant role in both operator theory and model theory of several domains; see, for instance, \cite{Bhat, NF}. In Section \ref{Pure-Isometry}, we focus our study on pure $\cP$-isometries. We provide a model for a certain class of pure $\cP-$isometries. We also present some examples of pure $\cP$-isometry.
	
	The Beurling–Lax–Halmos theorem connects operator theory with function theory on $\mathbb{D}$. It states that a non-zero closed subspace $M$ of the vector-valued Hardy space $H^2\otimes\mathcal{E}$ is invariant under the shift $M_z$ on $H^2\otimes\mathcal{E}$ if and only if there exist a Hilbert space $\mathcal{E}'$ and an isometric multiplier $\Theta$ such that $M$ is equal to $\Theta (H^2\otimes\mathcal{E}')$, see \cite{NF} as well as recent work on this topic in \cite{Curto2}. This theorem is known for various other cases, for example, complete Nevanlinna–Pick spaces \cite{MT}, space with a complete Nevanlinna–Pick factor \cite{CHS}, in the symmetrized bidisc \cite{Jaydeb}, in the tetrablock \cite{Sau} and in the symmetrized polydisc \cite{sub}. Section \ref{BLH-Rep} provides a Beurling-Lax-Halmos type theorem characterizing joint invariant subspaces of pure $\cP-$isometries.

	A paper  \cite{Pal-Penta} was announced on arXiv very recently. It has some overlap with this paper. In particular, Theorem \ref{P-unitary} and Theorem \ref{P-isometry} also appear there. However, the proof of Theorem \ref{P-isometry} in this paper is significantly different.

	\section{Preliminaries}
	This section provides an overview of the pentablock as well as pentablock-contractions. One of the main results of \cite{ALY-JMAA} contains several characterizations of a point to be in $\overline{\cP}$. In what follows, $\Gamma$ is the closure of the symmetrized bidisc $\mathbb{G}$. 
	
	\begin{thm}[Agler--Lykova--Young]
		Let 
		$$(s,p) = (\beta+\overline{\beta}p, p) = (\lambda_{1} + \lambda_{2}, \lambda_{1} \lambda_{2}) \in \Gamma$$ 
		where $|\beta|\leq 1$ and if $|p|=1$, then $\beta=\frac{s}{2}$. Let $a\in\mathbb{C}$. Then, the following statements are equivalent.
		\begin{enumerate}
			\item $(a,s,p) \in \overline{\cP},$ 
			\item $|a| \leq \bigg| 1 - \frac{\frac{1}{2} s \overline{\beta} }{ 1 + \sqrt{ 1 - | \beta |^{2}}} \bigg|,$
			\item $|a| \leq \frac{1}{2} | 1 - \overline{\lambda_{2}} \lambda_{1} | + \frac{1}{2} (1 - | \lambda_{1}|^{2})^{\frac{1}{2}} \frac{1}{2} (1 - | \lambda_{2}|^{2})^{\frac{1}{2}},$ and
			\item $\sup\limits_{z \in \mathbb{D}} | \Psi_{z}(a,s,p)| \leq 1, $ where $\Psi_{z}$ is the linear fractional map 
			$$\Psi_{z}(a,s,p) = \frac{a (1-|z|^{2})}{1 - sz + pz^{2}}.$$
		\end{enumerate} 
	\end{thm}
	
	The distinguished boundary of the pentablock is the Shilov boundary with respect to the uniform algebra $A(\cP),$ collection of all continuous functions on $\overline{\cP}$ which are holomorphic in $\cP.$ We shall denote it by $b\mathcal{P}$. The distinguished boundary of domains plays a distinguished role in studying complex geometry and the operator theory of domains. The distinguished boundary of the symmetrized bidisc is 
	$$ b  \Gamma = \{ (s,p) \in \Gamma : s = \overline{s} p \text{ and } |p|=1 \}.$$ 
	The following theorem gives descriptions of points in $b\cP$. 
	
	\begin{thm}\label{dist}
		For $(a,s,p)\in\mathcal{P},$ the following are equivalent:
		\begin{enumerate}
			\item $(a,s,p)\in b \cP$,
			\item $(s,p)\in b\Gamma$ and $|a|^2= 1-\frac{|s|^2}{4}$, 
			\item There exists a unique unitary matrix $U = \begin{pmatrix}
				u_{11} & u_{12} \\ u_{21} & u_{22}
			\end{pmatrix}$ such that
			$$u_{11} = u_{22} \hspace{5mm} \text{and} \hspace{5mm} (a,s,p) = (u_{21}, \operatorname{tr}(U), \operatorname{det}(U)), \text{ and }$$
			\item $(s, p)\in b\Gamma$ and the set $\{\lambda: (\lambda a, s, p)\in \overline{\cP}\}$ is equal to $\overline{\mathbb{D}}$.
		\end{enumerate}
	\end{thm}
	\begin{proof}
		The equivalence of statements $(1)$ and $(2)$  can be found in \cite{ALY-JMAA} and the equivalence of $(2)$ and $(3)$ can be found in \cite{Kumar}. We shall prove the equivalence of $(2)$ and $(4)$. This follows from the obeservation that if $(s,p)\in b \Gamma,$ then $(\lambda a, s, p) \in \overline{\cP}$ if and only if $|\lambda a|^{2} \leq 1 - \frac{|s|^{2}}{4}$. 
	\end{proof}

	Recall that a triple $(T_1, T_2, T_3)$ of commuting bounded operators on a Hilbert space is said to be a pentablock contraction ($\mathcal{P}$-contraction) if the following hold:
	\begin{enumerate}
		\item $$\sigma_{T}(T_1, T_2, T_3)\subset\overline{\mathcal{P}};$$
		\item $$\|r(T_1, T_2, T_3)\|\leq \|r\|_{\infty, \overline{\mathcal{P}}},$$
	\end{enumerate}
	for any rational function $r$ with poles outside $\overline{\mathcal{P}}.$ 
	
	Clearly if $(T_{1},T_{2},T_{3})$ is a $\cP-$contraction, then $(T_{2},T_{3})$ is a $\Gamma-$contraction. Conversely, if $(T_1, T_2	)$ is a $\Gamma$- contraction, then $(0, T_1, T_2)$ is a $\mathcal{P}$-contraction. The polynomial convexity of the pentablock gives the following. 
	\begin{lemma}
		Let $(T_1, T_2, T_3)$ be a triple of commuting bounded operator on a Hilbert space $\cH$. Then, $(T_1, T_2, T_3)$ is a pentablock contraction if and only if
		$$\|f(T_1, T_2, T_3)\|\leq \|f\|_{\infty, \overline{\mathcal{P}}}$$ for any holomorphic polynomial $f$ in three variables.
	\end{lemma}

	Note that the restriction of a pentablock contraction on a joint invariant subspace is again a pentablock contraction. Indeed, if $\cM$ is a joint invariant subspace of a pentablock contraction $(T_1, T_2, T_3)$ and $f$ is any polynomial in three variables, then
	\begin{align*}
		\|f(T_1|_{\cM}, T_{2}|_{\cM}, T_{3}|_{\cM})\|&=\|f(T_1, T_2, T_3)|_{\cM}\|\leq \|f(T_1, T_2, T_3)\|\leq \|f\|_{\infty, \overline{\mathcal{P}}}.
	\end{align*}
	Similarly, we can show that the adjoint of a pentablock contraction is also a pentablock contraction. Below, we shall give some more examples of pentablock contractions.
	\begin{enumerate}
		\item Let $T_1$ and $T_2$ be two commuting contractions. Note that $(z, 0, w)\in\cP$ whenever $(z,w)\in\mathbb{D}^2$. Let $\delta:\mathbb{D}^2\rightarrow\cP$ be a map defined as 
		$$(z,w)\mapsto ( z, 0, w),$$ and let $f$ be any polynomial. Then
		$$\|f(T_1, 0, T_2)\|=\|f\circ\delta(T_1, T_2)\|\leq \|f\circ\delta\|_{\infty, \mathbb{D}^2}\leq\|f\|_{\infty, \cP}.$$
		Thus, $(T_1, 0, T_2)$ is a pentablock contraction.
		\item Let $(T_{1},T_{2},T_{3})$ be a commuting triple of normal operators. Suppose $\sigma_{T}(T_{1},T_{2},T_{3}) \subseteq \overline{\cP}$. Then $(T_{1}, T_{2},T_{3})$ is a pentablock contraction.
	\end{enumerate}
	We shall use some spaces of vector-valued and operator-valued functions. Let $\cE$ be a separable Hilbert space. Let $H^2(\cE)$ denote the usual Hardy space
	of analytic $\cE$-valued functions on $\bD$ and $L^2(\cE)$ the Hilbert space of square-integrable $\cE$-valued functions on $\bT$, with their natural inner products. We shall denote the space of bounded
	analytic $\mathcal{B}(\cE)$-valued functions on $\bD$ by $H^\infty(\mathcal{B}(\cE))$ and the space of bounded measurable $\mathcal{B}(\cE)$-valued
	functions on $\bT$,  by $L^\infty(\mathcal{B}(\cE))$, each with the appropriate version of the supremum norm.
	In the following lemmas, we observe some basic facts about pentablock contractions. We want to mention that we shall not use these lemmas anywhere in this note. We are writing it here just for the record.
	\begin{lemma}
		If $(T_1, T_2, T_3)$ is a $\cP$-contraction, then for all $z\in\overline{\mathbb{D}}$ the following hold:
		\begin{enumerate}
			\item $(zT_{2}, z^{2} T_{3})$ is a $\Gamma $-contraction,	
			\item $I - (1 - |z|^{2})^{2} T_{1}^{*}T_{1} + |z|^{2} T_{2}^{*} T_{2} + |z|^{4} T_{3}^{*} T_{3}  - 2 Re ( z T_{2} - z^{2} T_{3} + z |z|^{2} T_{2}^{*} T_{3} )   \geq{0}$.\label{P2}
		\end{enumerate}
	\end{lemma}
	\begin{proof}
		The first part follows from the observation that if $(s, p)\in\Gamma$, then $(zs, z^2p)$ also is in $\Gamma$ for all $z\in\overline{\mathbb{D}}$. Now, we shall prove part $\eqref{P2}$. Since the triple $(T_1, T_2, T_3)$ is a pentablock contraction, it follows that the operators  
		$$\Psi_{z}(T_1,T_2,T_3):= (1-|z|^{2})T_1\left(I-zT_2+z^2T_3\right)^{-1}$$
		are contractions for all $z\in\overline{\mathbb{D}}$. Thus for all $z\in\overline{\mathbb{D}}$, we have 
		\begin{align*}
			0&\leq I-\Psi_{z}(T_1,T_2,T_3)^*\Psi_{z}(T_1,T_2,T_3).
		\end{align*}
		A simple matrix calculation will show that the right hand of the equation above is 
		\begin{align*}
			A^*\left[I - (1 - |z|^{2})^{2} T_{1}^{*}T_{1} + |z|^{2} T_{2}^{*} T_{2} + |z|^{4} T_{3}^{*} T_{3}  - 2 Re ( z T_{2} - z^{2} T_{3} + z |z|^{2} T_{2}^{*} T_{3} )\right]A,
		\end{align*}
		where $A$ is some operator. Therefore the positivity of $$I-\Psi_{z}(T_1,T_2,T_3)^*\Psi_{z}(T_1,T_2,T_3)$$ gives the part \eqref{P2}. This completes the proof.
	\end{proof}
	
	\begin{lemma}
		Let $\varphi_1, \varphi_2,$ and $\varphi_3$  be functions in $H^\infty(\mathcal{B}(\cE)).$ Then the operator triple $\left(M_{\varphi_1}, M_{\varphi_2}, M_{\varphi_3}\right)$ on $H^{2}(\cE)$ is a $\cP$-contraction if and only
		if the operator triple $(\varphi_1(z), \varphi_2(z), \varphi_3(z))$ is a $\cP$-contraction for every $z\in \mathbb{D}$.
	\end{lemma}
	\begin{proof}
		By definition of the pentablock contraction, the triple $\left(M_{\varphi_1}, M_{\varphi_2}, M_{\varphi_3}\right)$ is a $\cP$-contraction if and only if 
		\begin{align}\label{PC1}
			\|f(M_{\varphi_1}, M_{\varphi_2}, M_{\varphi_3})\|\leq \|f\|_{\infty, \overline{\mathcal{P}}}
		\end{align}
		for any polynomial $f$ on $\mathcal{P}$. Using the fact that $f(M_{\varphi_1}, M_{\varphi_2}, M_{\varphi_3})=M_{f\left(\varphi_1,\varphi_2, \varphi_3\right)}$ and $\|M_{\varphi}\|= \|\varphi\|_{\infty, \mathbb{D}}$ for any $\varphi\in H^\infty(\mathcal{B}(\cE))$, we get that the equation \eqref{PC1} is true if and only if 
		$$\|f({\varphi_1}, {\varphi_2}, {\varphi_3})\|_{\infty, \mathbb{D}}  \bydef {\rm sup}\{  \|f({\varphi_1 (z)}, {\varphi_2(z)}, {\varphi_3(z)})\| : z \in \mathbb{D}\} \leq \|f\|_{\infty, \overline{\mathcal{P}}},$$
		for any polynomial $f$ on $\mathcal{P}$. This completes the proof.
	\end{proof}
	
	\section{Pentablock unitaries}\label{PU}
	The section aims to discuss several characterizations of pentablock unitaries. To commence our discussion, we will revisit the definition of pentablock unitaries.
	\begin{definition}
		A triple $(N_1, N_2, N_3)$ on a Hilbert space $\cH$ is said to be a pentablock unitary ($\mathcal{P}$-unitary) if the following hold:
		\begin{enumerate}
			\item
			$N_1, N_2, N_3$ are commuting normal operators, and
			\item
			$\sigma_{T}(N_1, N_2, N_3)\subset {b\mathcal{P}}$.
		\end{enumerate}
	\end{definition}
	Let us see an example of a pentablock unitary.
	\begin{example}
		Let $\beta\in\mathbb{T}$. Then, the commuting triple of normal operators  $\left(M_{\frac{\beta-\overline{\beta} z}{2}}, M_{\beta+\overline{\beta}z}, M_{z}\right)$ on $L^2(\bT)$ is a pentablock unitary. Indeed, set
		$$\varphi_1(z)= \frac{\beta-\overline{\beta}z}{2}, \quad \varphi_2(z)=\beta+\overline{\beta}z \quad\text{ and } \quad \varphi_3(z)=z.$$
		Note that, for $z\in\mathbb{T}$, we get $\varphi_{2}(z) = \overline{\varphi_{2}(z)} \varphi_{3}(z)$, $|\varphi_3(z)|=1,$  and $|\varphi_{2}(z)| \leq 2.$ Hence,  $(\varphi_{2}(z), \varphi_{3}(z)) \in b\Gamma$ for all $z\in\bT.$ Now, for $z\in\mathbb{T}$,
		\begin{align}\label{PX1}
			|\varphi_1(z)|^2 = \frac{1}{2}-\frac{1}{4}\left[\overline{\beta}^2z+\beta^2\overline{z} \right]
			\quad \text{ and } \quad 
			|\varphi_2(z)|^2 = 2+\beta^2\overline{z}+\overline{\beta}^2z.
		\end{align}
		Thus, from equation \eqref{PX1}, for all $z\in\mathbb{T}$, we have
		$$4|\varphi_1(z)|^2+|\varphi_2(z)|^2=4.$$
		Hence, $\left(\varphi_1(z), \varphi_2(z), \varphi_3(z)\right)\in b\cP$ for all $z\in\bT$. Since $\left(M_{\frac{\beta-\overline{\beta} z}{2}}, M_{\beta+\overline{\beta}z}, M_{z}\right)$ is a triple of commuting normal operators, it follows that the Taylor joint spectrum of the triple	$(M_{\frac{\beta-\overline{\beta} z}{2}}, M_{\beta+\overline{\beta}z}, M_{z})$ is contained in  ${b\mathcal{P}}$. Therefore, $\left(M_{\frac{\beta-\overline{\beta} z}{2}}, M_{\beta+\overline{\beta}z}, M_{z}\right)$ is a $\cP$-unitary.
	\end{example}
	The $\Gamma$-unitaries have been extensively studied, and their various characterizations are known, see \cite{AY, TPS-Adv}. For our discussion, we shall focus exclusively on the following ones.
	\begin{thm}\label{Gamma-unitary}
		Let $(S, P)$ be a pair of commuting operators defined on a Hilbert space. Then the following are equivalent:
		\begin{enumerate}
			\item $(S,P)$ is a $\Gamma-$unitary,
			\item $P$ is a unitary, $S = S^{*}P,$ and $r(S) \leq 2,$ where $r(S)$ is the spectral radius of $S$, and
			\item $P$ is a unitary, and $(S,P)$ is a $\Gamma-$isometry.
		\end{enumerate}
	\end{thm}
	
	Once we have a description of the distinguished boundary, we can anticipate the characterization of the pentablock unitaries. With Theorem \ref{dist} in hand, we have the following characterizations of the pentablock unitaries. Ideas in this proof are taken from \cite{Tirtha-Tetrablock}.
	\begin{thm}\label{P-unitary}
		Let $(N_1, N_2, N_3)$ be a commuting triple of bounded operators on a Hilbert space $\cH$. Then the following are equivalent:
		\begin{enumerate}
			\item $(N_1, N_2, N_3)$ is a $\mathcal{P}$-unitary;
			\item $N_{1}$ is  normal, $(N_2, N_3)$ is a $\Gamma$-unitary, and $N_{1}^{*} N_{1} = I -\frac{1}{4}N_2^{*}N_2$;
			\item There are commuting normal operators $U_1, U_2$ and $U_3$ such that the $2 \times 2$ block operator matrix $U = \left(
			\begin{array}{cc}
				U_1 & U_2 \\
				U_3 & U_1 \\
			\end{array}
			\right)$
			is a unitary operator and $(N_1,N_2, N_3) = (U_3, 2U_1, U_1^2 - U_2U_3)$. 
		\end{enumerate}
	\end{thm}
	\begin{proof} We establish the equivalence of $(1), (2)$ and $(3)$ by showing that
		$$ (3)\implies (2)\implies (1) \implies (2)\implies(3).$$
		
		$\text {Proof of } (3)\Rightarrow (2):$  Set $N_1= U_3, N_2= 2U_1$ and $N_3= U_1^2-U_2U_3$. Since $U_1, U_2$, and $U_3$ are commuting normal operators, so $N_1, N_2$ and $N_3$ are commuting normal operators. First we show that $(N_2, N_3)$ is a $\Gamma$-unitary. Since  $U$ is unitary, so we have 
		$$ \left(
		\begin{array}{cc}
			U_1 & U_2 \\
			U_3 & U_1 \\
		\end{array}
		\right)\left(
		\begin{array}{cc}
			U_1^* & U_3^* \\
			U_2^* & U_1^* \\
		\end{array}
		\right)=I=\left(
		\begin{array}{cc}
			U_1^* & U_3^* \\
			U_2^* & U_1^* \\
		\end{array}
		\right)\left(
		\begin{array}{cc}
			U_1 & U_2 \\
			U_3 & U_1 \\
		\end{array}
		\right).$$
		A simple matrix computation will give the following: 
		\begin{equation}\label{u1}
			U_1U_1^*+U_2U_2^*=I= U_3U_3^*+U_1U_1^*,  \quad U_1U_3^*+U_2U_1^*=0 
		\end{equation}
		\begin{equation}\label{u2}
			U_1^*U_1+U_3^*U_3=I= U_2^*U_2+U_1^*U_1,  \quad U_1^*U_2+U_3^*U_1=0.
		\end{equation}
		Using equations \eqref{u1} and \eqref{u2}, we get 
		$$ N_2^*N_3= 2U_1^*\left(U_1^2- U_2U_3\right) = 2U_1 =N_2. $$
		Also with the help of equations \eqref{u1}, \eqref{u2} and fact that $N_1, N_2$ and $N_3$ are commuting normal operators, we have 
		\begin{align*}
			N_3^*N_3&= \left({U_1^*}^2-U_3^*U_2^*\right)\left(U_1^2-U_2U_3\right)\\
			&={U_1^*}^2U_1^2+{U_1^*}U_3^*U_1U_3+U_3^*U_1^*U_3U_1+U_3^*\left(I-U_1^*U_1\right) U_3\\
			&={U_1^*}^2U_1^2+{U_1^*}U_1\left(I-U_1^*U_1\right)+U_3^*U_3 = I.
		\end{align*}
		Similarly, we can also show that $N_3N_3^*=I$. Thus, by Theorem \ref{Gamma-unitary}, the pair $(N_2, N_3)$ is a $\Gamma-$ unitary. Now, 
		$$I-\frac{N_2N_2^*}{4}= I-U_1U_1^*=U_3U_3^*=N_1N_1^*=N_1^*N_1.$$ 
		Therefore $(2)$ holds.
		
		$\text {Proof of }(2)\Rightarrow(1):$ The operators $N_1, N_2$ and $N_3$ are commuting normal by assumption. Consider the $C^*$-algebra $\mathcal{C}$ generated by the operators $N_1, N_3$, and $N_3$. Since $\mathcal{C}$ is commutative, it is isometrically isomorphic to $C\left(\sigma_{T}(N_1. N_2, N_3)\right)$ by the Galfand map. The inverse of the Gelfand map takes coordinate functions $x_j$ to $N_j.$ We need to show that the Taylor joint spectrum of the triple $(N_1, N_2, N_3)$ is contained in the distinguished boundary of the pentablock. To that end, let $(x_1, x_2, x_3)\in \sigma_{T}(N_1, N_2, N_3)$. Using the relation, $N_1^{*} N_1 = I - \frac{N_2^{*}N_2}{4}$ and the fact that the inverse Gelfand sends $x_j$ to $N_j$, we get 
		$$|x_1|^2=1-\frac{|x_2|^2}{4} \quad \text{ on } b\mathcal{P}.$$
		Also, since $(N_2, N_3)$ is a $\Gamma$-unitary, so $(x_1, x_2)\in b\Gamma$. Hence, by Therem \ref{dist} $(x_1, x_2, x_3)\in b\mathcal{P}$.
		
		$\text {Proof of }(1)\Rightarrow(2):$ Suppose $(N_1, N_2, N_3)$ is a $\mathcal{P}$- unitary. Then $N_1, N_2$, and $N_3$ are commuting normal operators. Therefore, the ${C}^*$ algebra generated by  $N_1, N_2$, and $N_3$, say $\mathcal{C}$, is commutative. Thus, by the Gelfand theory, $\mathcal{C}$ is isometrically isomorphic, via the Gelfand map,  to $C\left(\sigma_{T}(N_1. N_2, N_3)\right)$. Moreover, the Gelfand map sends $N_j$ to the coordinate functions $x_j.$ Any point $(x_1, x_2, x_3)$ on the distinguished boundary of $\mathcal{P}$ satisfies the following relations:
		\begin{equation}\label{distpoint}
			|x_{1}|^{2} = 1 - \frac{|x_{2}|^{2}}{4}, \quad x_{2} = \overline{x_{2}} x_{3}, \quad |x_{3}| =1 \quad \text{and} \quad |x_{2}| \leq 2.
		\end{equation}
		Now $(2)$ follows by observing that \eqref{distpoint} holds for the operators $N_{1},N_{2},N_{3}.$
		
		Proof of $(2) \Rightarrow (3):$ Set $U_{1} = \frac{1}{2} N_{2}, U_{2} = - N_{1}^{*} N_{3}, \text{ and } U_{3} = N_{1}.$ Clearly,  $U_{1}, U_{2},$ and $U_{3}$ are commuting normal operators. A direct matrix computation will give that the block operator matrix $U = \begin{bmatrix}
			U_{1} & U_{2} \\
			U_{3} & U_{1}
		\end{bmatrix}$ 
		is unitary. 
	\end{proof}
	\section{Wold-type decomposition for pentablock isometries}\label{WD}
	
	As the name suggests, this section aims to discuss pentablock isometries and the structure of pentablock isometries, i.e., Wold-type decomposition for pentablock isometries. As mentioned in the introduction, an additional term appears in the Wold-type decomposition. We will provide a functional model for that term. We shall start this section with the following definition. 
	\begin{definition}
		A commuting triple $(V_1, V_2, V_3)$ on $\cH$ is called a $\cP$-isometry if there exists a Hilbert space $\cK$ containing $\cH$ and a $\cP$-unitary $(\tilde{V_1}, \tilde{V_2}, \tilde{V_3})$ on $\cK$ such that $\cH$ is a joint invariant subspace of $(\tilde{V_1}, \tilde{V_2}, \tilde{V_3})$ and
		$$ V_1= \tilde{V_1}|_{\cH},  \quad V_3= \tilde{V_3}|_{\cH} \quad and \quad V_3= \tilde{V_3}|_{\cH}.$$
		In other words, a pentablock isometry is the restriction of a pentablock unitary to a joint invariant subspace. 
	\end{definition}
	Note that if $(V_{1},V_{2},V_{3})$ is a $\cP-$isometry, then $V_{3}$ is an isometry. Let us observe two basic examples of $\cP-$isometries. 
	\begin{enumerate}
		\item Let $V_1$ and $ V_2$ be two commuting isometries. Then $(V_1, 0, V_2)$ is a $\cP-$ isometry.
		\item Let $V$ be an isometry. Then $(0, 2V, V)$ is a $\cP-$ isometry.
	\end{enumerate}
	
	The following lemma gives a necessary condition for a triple to be a pentblock isometry.  
	\begin{lemma}\label{in}
		Let $(V_{1},V_{2},V_{3})$ be a $\cP-$isometry on a Hilbert space $\cH.$ Then $(V_{2},V_{3})$ is a $\Gamma-$isometry and 
		$$ V_{1}^{*} V_{1} = I - \frac{1}{4} V_{2}^{*}V_{2}.$$
	\end{lemma}
	\begin{proof}
		Since $(V_{1},V_{2},V_{3})$ is a $\cP-$ isometry on $\cH,$ there exists a Hilbert space $\cK \supseteq \cH$ and a $\cP-$unitary $(N_{1},N_{2},N_{3})$ on $\cK$ such that $\cH$ is a joint invariant subspace for each $N_{i}$ and $N_{i}|_{\cH} = V_{i}.$ Clearly, $(N_{2},N_{3})$ is a $\Gamma-$unitary. Hence $(V_{2},V_{3})$ is a $\Gamma-$isometry. The equality $N_{1}^{*} N_{1} = I - \frac{1}{4} N_{2}^{*} N_{2}$ gives $V_{1}^{*}V_{1} = I - \frac{1}{4} V_{2}^{*} V_{2}.$	
	\end{proof}
	There is a natural question to ask here: Whether the converse of Lemma \ref{in} holds true or not. In section \ref{P-Iso}, we answer this question affirmatively.
	
	If $(V_{1},V_{2},V_{3})$ is a $\cP-$isometry and $V_{3}$ is unitary, then by Lemma \ref{in} and Theorem \ref{Gamma-unitary}, $(V_{2},V_{3})$ is a $\Gamma-$unitary. Moreover, $V_{1}$ commutes with $V_{1}^{*}V_{1}.$ Indeed, 
	$$V_{1} V_{1}^{*}V_{1} = V_{1} - \frac{1}{4} V_{1} V_{2}^{*}V_{2} = V_{1} - \frac{1}{4} V_{2}^{*}V_{2}V_{1} = V_{1}^{*} V_{1}^{2}.$$ This motivates us to make the following definition:
	
	\begin{definition}
		A $\cP-$isometry $(V_{1}, V_{2}, V_{3})$ is called a quasi $\cP-$unitary if $V_{3}$ is a unitary operator. 
	\end{definition}
	The Wold-type decomposition for $\cP-$ isometries is as follows. 
	\begin{thm}
		Let $(V_{1},V_{2},V_{3})$ be a $\cP-$isometry on a Hilbert space $\cH.$ Then there is a decomposition of $\cH$ into a direct sum $\cH = \cH_{1} \oplus \cH_{2}$ satisfying the following conditions:
		\begin{enumerate}
			\item the subspaces $\cH_{1}$ and $\cH_{2}$ are reducing subspaces for each $V_{i},$
			\item if $R_{i} = V_{i}|_{\cH_{1}},$ then the triple $(R_{1}, R_{2}, R_{3})$ is a quasi $\cP-$unitary, and
			\item  if $W_{i} = V_{i}|_{\cH_{2}},$ then the triple $(W_{1}, W_{2}, W_{3})$ is a pure $\cP-$ isometry.
		\end{enumerate}
	\end{thm}
	\begin{proof}
		Since $(V_{1},V_{2},V_{3})$ is a $\cP-$isometry, $(V_{2},V_{3})$ is a $\Gamma-$isometry. Applying Wold-type decomposition for $\Gamma-$isometry $(V_{2},V_{3}),$ we get a decomposition of $\cH$ into a direct sum $\cH = \cH_{1} \oplus \cH_{2}$ and with respect to this decomposition $V_{1}$ and $V_{2}$ have the following block matrix representation
		$$V_{2} = \begin{bmatrix}
			R_{2} & 0 \\ 0 & W_{2}
		\end{bmatrix} \quad \text{ and } \quad V_{3} = \begin{bmatrix}
			R_{3} & 0 \\ 0 & W_{3}
		\end{bmatrix}$$
		where $(R_{2},R_{3})$ is a $\Gamma-$unitary on $\cH_{1}$ and $(W_{2}, W_{3})$ is a pure $\Gamma-$isometry on $\cH_{2},$ see Theorem 2.6 in \cite{AY}. Note that $R_{3}$ is unitary, and $W_{3}$ is a pure isometry. If we now write $V_{1}$ according to this decomposition of $\cH,$ then we let its block matrix be 
		$$V_{1} = \begin{bmatrix}
			V_{11} & V_{12} \\ V_{21} & V_{22}
		\end{bmatrix}.$$ 
		By commutativity of $V_{1}$ and $V_{3},$ we get that $V_{12} W_{3} = R_{3}V_{12}$ and  $V_{21} R_{3} = W_{3} V_{21}.$ Thus, $W_{3}^{*n} V_{12}^{*} = V_{12}^{*} R_{3}^{*n}$ and $R_{3}^{*n}V_{21}^{*} = V_{21}^{*} W_{3}^{*n}$ for any $n\geq 1.$ Since $W_{3}$ is a pure isometry, $W_{3}^{*n}$ converges to $0$ strongly. Therefore, $V_{12} = V_{21} = 0.$ Set $R_{1} = V_{11}$ and $W_{1} = V_{22}.$ Clearly, both the triples $(R_{1},R_{2},R_{3})$ and $(W_{1},W_{2},W_{3})$ are $\cP-$isometries. The rest of the proof follows by noting that $R_{3}$ is unitary and $W_{3}$ is a pure isometry. 
	\end{proof}
	\subsection{Characterization of quasi $\cP-$unitaries} 
	The lemma below gives a class of examples which are quasi $\cP-$unitary but not $\cP$-unitary.
	\begin{lemma}
		Let $(N_{1},N_{2},N_{3})$ be a $\cP-$unitary on a Hilbert space $\cE.$ Then the triple $(M_{z} \otimes N_{1}, I \otimes N_{2}, I \otimes N_{3})$ on $H^{2} \otimes \cE$ is quasi $\cP-$unitary.
	\end{lemma}
	\begin{proof}
		It is enough to show that  $(M_{z} \otimes N_{1}, I \otimes N_{2}, I \otimes N_{3})$ is a $\cP-$isometry because $I \otimes N_{3}$ is a unitary. Consider the multiplication operator $\tilde{M_{z}},$ multiplication by $z$ on $L^{2}(\mathbb{T})$. Then by Theorem \ref{P-unitary}, the triple ($\tilde{M_z}\otimes N_1, I\otimes N_2, I\otimes N_3)$ on Hilbert space $L^2\otimes\cE$ is a $\cP-$unitary. The restriction of this $\cP-$unitary to $H^{2} \otimes \cE$ is our triple  $(M_{z} \otimes N_{1}, I \otimes N_{2}, I \otimes N_{3}).$ This proves that the given triple is a $\cP-$isometry. 
	\end{proof}
	The following theorem provides a complete characterization of quasi $\cP-$unitaries.
	\begin{thm}\label{quasi}
		Let $(R_{1},R_{2},R_{3})$ be a quasi $\cP-$unitary on a Hilbert space $\cH.$ Then the triple $(R_{1},R_{2},R_{3})$ is unitarily equivalent to the commuting triple 
		\begin{equation} \label{eq}
			\left( \begin{bmatrix}
				U_{1} & 0 \\ 0 & M_{z} \otimes N_{1}
			\end{bmatrix}, \begin{bmatrix}
				U_{2} & 0 \\ 0 & I \otimes N_{2} 
			\end{bmatrix}, \begin{bmatrix}
				U_{3} & 0 \\ 0 & I \otimes N_{3}
			\end{bmatrix}\right)
		\end{equation} on $\cE \oplus (H^{2} \otimes \cF)$ for some Hilbert spaces $\cE$ and $\cF,$ where $(U_{1}, U_{2}, U_{3})$ and $(N_{1},N_{2},N_{3})$ are $\cP-$unitaries on Hilbert spaces $\cE$ and $\cF$ respectively. 
	\end{thm}
	\begin{proof}
		First, we shall prove that ${\rm Ker}R_{1}$ is a reducing subspace for $R_{1},R_{2}$ and $R_{3}.$ Trivially, ${\rm Ker} R_{1}$ is invariant under $R_{1}.$ Since $R_{1}$ commutes with $R_{1}^{*}R_{1},$ we get $R_{1}^{*}R_{1}R_{1}^{*} = R_{1}^{*2}R_{1}.$ It follows that for $h \in {\rm Ker} R_{1},$ $R_{1}^{*}h \in {\rm Ker} R_{1}^{*} R_{1} = {\rm Ker} R_{1}.$ Since $R_{1}$ commutes with $R_{2}$ and $R_{3},$ ${\rm Ker} R_{1}$ is invariant under $R_{2}$ and $R_{3}.$ Normality of $R_{2}$ and $R_{3}$ implies that $R_{1}$ commutes with $R_{2}^{*}$ and $R_{3}^{*}.$ Hence ${\rm Ker} R_{1}$ is invariant under $R_{2}^{*}$ and $R_{3}^{*}.$ This shows that ${\rm Ker} R_{1}$ is reducing for $R_{1},R_{2}$ and $R_{3}.$ Clearly, the triple $\left(R_{1}|_{{\rm Ker}R_{1}}, R_{2}|_{{\rm Ker}R_{1}},R_{3}|_{{\rm Ker}R_{1}} \right)$ is a $\cP-$unitary on ${\rm Ker}R_{1}.$
		
		From the discussion above, it is enough to prove the result for a quasi $\cP-$unitary $(R_{1}, R_{2}, R_{3})$ with $R_{1}$ injective. Let $V|R_{1}|$ be the polar decomposition of $R_{1}$ where $V$ is an isometry and $|R_{1}| = (R_{1}^{*}R_{1})^{1/2}.$ Since $R_{1}$ commutes with $R_{1}^{*}R_{1},$ 
		the isometry $V$ commutes with $|R_{1}|.$ The commutativity of $R_{1}$ with $R_{2}$ and the commutativity of $V$ with $|R_{1}|$ gives us the following
		$$|R_{1}| (VR_{2} - R_{2} V) = 0.$$
		The injectivity of $R_{1}$ will give us that $VR_{2} = R_{2}V.$ Similarly, we can show that $VR_{3} = R_{3}V.$ Let $V = U \oplus W$ be the Wold decomposition of isometry $V$ into its unitary part $U$ and the shift part $W.$ Suppose $\cH = \cH_{1} \oplus \cH_{2}$ is the corresponding decomposition of the whole space $\cH.$ Thus, $V$ has the following block matrix decomposition 
		$$ \begin{bmatrix}
			U & 0 \\ 0 & W
		\end{bmatrix} .$$ 
		If we now write $R_{2}$ according to this decomposition of $\cH,$ then we let its block matrix be 
		$$
		R_{2} = \begin{bmatrix}
			X_{11} & X_{12} \\ X_{21} & X_{22}
		\end{bmatrix}.$$
		By commutativity of $R_{2}$ and $V,$ we get that $X_{12} W = U X_{12}$ and $X_{21} U = W X_{21}.$ Thus, $W^{*n} X_{12}^{*} = X_{12}^{*}U^{*n}$ and $U^{*n} X_{21}^{*} = X_{21}^{*}W^{*n}$ for any $n\geq 1.$ Since $W$ is a shift, $W^{*n}$ converges to $0$ strongly. Therefore, $X_{12} = X_{21} = 0.$ Hence, $R_{2}$ is a block diagonal matrix.
		$$R_{2} = \begin{bmatrix}
			X_{11} & 0 \\ 0 & X_{22}
		\end{bmatrix}.$$
		Similarly, we can show that $|R_{1}|$ and $R_{3}$ can be written in the following matrix form
		$$|R_{1}| = \begin{bmatrix}
			P_{11} & 0 \\ 0 & P_{22} 
		\end{bmatrix} \quad \text{and} \quad R_{3} = \begin{bmatrix}
			Y_{11} & 0 \\ 0 & Y_{22}
		\end{bmatrix}$$
		With respect to the decomposition  $\cH = \cH_{1} \oplus \cH_{2}.$ Now, $R_{1}$ takes the form. 
		$$R_{1} = V |R_{1}|= \begin{bmatrix}
			U P_{11} & 0 \\ 0 & W P_{22}
		\end{bmatrix}.$$
		Clearly, both the triples $(UP_{11}, X_{11}, Y_{11})$ on $\cH_{1}$ and $(WP_{22}, X_{22}, Y_{22})$ on $\cH_{2}$ are quasi $\cP-$unitaries. Moreover, $U$ commutes with $P_{11},$ and $W$ commutes with $P_{22}.$ Since $U$ is unitary and $P_{11}$ is a positive operator, we get that $UP_{11}$ is a normal operator on $\cH_{1}.$ This proves that the triple $(UP_{11}, X_{11}, Y_{11})$ on $\cH_{1}$ is a $\cP-$unitary. Now we shall prove that the triple $(WP_{22}, X_{22}, Y_{22})$ on $\cH_{2}$ is unitarily equivalent to the triple $(M_{z} \otimes N_{1}, I \otimes N_{2}, I \otimes N_{3})$ on $H^{2} \otimes \cF$ for some Hilbert space $\cF,$ where $(N_{1},N_{2},N_{3})$ is a $\cP-$unitary on $\cF.$ Since $W$ is a pure isometry, there exists a Hilbert space $\cF$ such that $W$ is unitarily equivalent to $M_{z} \otimes I$ on $H^{2} \otimes \cF.$ It follows from the injectivity of $R_{1}$ that $W$ and $P_{22}$ commute with both $X_{22}$ and $Y_{22}.$ Thus, $(P_{22}, X_{22}, Y_{22})$ is a commuting triple of normal operators on $\cH_{2}$, which commute with the pure isometry $W.$ It follows that there exists a commuting triple of normal operators $(N_{1}, N_{2}, N_{3})$ on $\cF$ such that the triple $(P_{22}, X_{22}, Y_{22})$ is unitarily equivalent to the triple $(I \otimes N_{1}, I \otimes N_{2}, I \otimes N_{3})$ on $H^{2} \otimes \cF.$ Therefore, the triple $(UP_{22}, X_{22}, Y_{22})$ on $\cH_{2}$ is unitarily equivalent to the triple $(M_{z} \otimes N_{1}, I \otimes N_{2}, I \otimes N_{3})$ on $H^{2} \otimes \cF.$ Since $(UP_{22}, X_{22}, Y_{22})$ is a $\cP-$isometry, we get that 
		$$(M_{z} \otimes N_{1})^{*} (M_{z} \otimes N_{1}) = I - \frac{1}{4} (I \otimes N_{2})^{*} (I \otimes N_{2}).$$
		This implies that 
		$$N_{1}^{*}N_{1} = I - \frac{1}{4} N_{2}^{*}N_{2}.$$
		We shall apply Theorem \ref{P-unitary} to conclude that the triple $(N_{1},N_{2},N_{3})$ is  a $\cP-$unitary. This completes the proof.
	\end{proof}
	We have the following result, as an application of Theorem \ref{quasi}.
	\begin{corollary}\label{quasi cor}
		Let $(R_{1},R_{2},R_{3})$ be a commuting triple of bounded operators. Then the following are equivalent:
		\begin{enumerate}
			\item $(R_{1},R_{2},R_{3})$ is a quasi $\cP-$unitary;
			\item $(R_{2},R_{3})$ is a $\Gamma-$unitary and $R_{1}^{*}R_{1} = I - \frac{1}{4} R_{2}^{*}R_{2}.$
		\end{enumerate}
	\end{corollary}
	\begin{proof}
		We only need to show that $(2)$ implies $(1).$ It is clear from the proof of Theorem \ref{quasi} that a commuting triple $(R_{1},R_{2},R_{3})$ satisfying $(2)$ is unitarily equivalent to the commuting triple in \eqref{eq}. Therefore, $(R_{1},R_{2},R_{3})$ is a quasi $\cP-$unitary.
	\end{proof}
	
	\section{Characterization of $\cP-$ isometries}\label{P-Iso}
	This section aims to provide an algebraic characterization of $\cP-$isometries. We have already discussed that if $(V_1, V_2, V_3)$ is a $\cP$-isometry, then $(V_2, V_3)$ is a $\Gamma$-isometry. There are many established characterizations for determining when a pair qualifies as a $\Gamma$-isometry, see \cite{AY-model, TPS-Adv}. Nevertheless, we will use the subsequent algebraic characterization for these pairs.
	
	\begin{thm}\label{Gamma-isomtery}
		Let $(S, P)$ be a pair of commuting operators on a Hilbert
		space $\cH$. Then the following are equivalent:
		\begin{enumerate}
			\item  The pair $(S, P)$ is a $\Gamma$-isometry;
			\item $S=S^*P, \quad  P^*P=I$\quad and \quad $\|S\|\leq 2$. 
			
		\end{enumerate}
	\end{thm}
	
	The algebraic characterization of pentablock isometries is as follows:
	\begin{thm}\label{P-isometry}
		Let $(V_1, V_2, V_3)$ be a commuting triple of bounded operators on a Hilbert space $\cH$. Then the following are equivalents: 
		\begin{enumerate}
			\item $(V_1, V_2, V_3)$ is a $\cP$-isometry;
			\item $(V_2, V_3)$ is a $\Gamma$-isometry and $V_1^*V_1=I-\frac{1}{4}V_2^*V_2$.
		\end{enumerate}
	\end{thm}
	\begin{proof}
		The $(1) \Rightarrow (2)$ is Lemma \ref{in}. We only need to show that $(2) \Rightarrow (1).$ 
		
		
		Since $(V_{2}, V_{3})$ is a $\Gamma-$isometry, there exists a Hilbert space $\tilde{\cH} \supseteq \cH$ and a $\Gamma-$unitary $(\tilde{V_{2}}, \tilde{V_{3}})$ on $\tilde{\cH}$ such that 
		$\tilde{V_{i}}|_{\cH} = V_{i}$ for $i=2,3$ where 
		$$\tilde{\cH} = \overline{\rm span} \{ \tilde{V_{2}}^{*n} \tilde{V_{3}}^{*m} h: n,m \geq 0 \text{ and } h\in\cH\}.$$
		That is, $(\tilde{V_{2}}, \tilde{V_{3}})$ is the minimal $\Gamma-$unitary extension of the $\Gamma-$isometry $(V_{2}, V_{3}).$ Now we extend the triple $(V_{1},V_{2},V_{3})$ to a quasi $\cP-$unitary. To do that, we define $\tilde{V_{1}}$ on $\tilde{\cH}$ such that 
		\begin{enumerate}
			\item[(a)] $\tilde{V_{1}}$ is an extension of $V_{1},$ i.e., $\tilde{V_{1}}|_{\cH} = V_{1},$
			\item[(b)] $\tilde{V_{1}}$ commutes with $\tilde{V_{2}}$ and $\tilde{V_{3}},$ and
			\item[(c)] $\tilde{V_{1}}^{*} \tilde{V_{1}} = I_{\tilde{\cH}} - \frac{1}{4} \tilde{V_{2}}^{*} \tilde{V_{2}}.$
		\end{enumerate}
		Define $$\tilde{V_{1}} (\tilde{V_{2}}^{*n} \tilde{V_{3}}^{*m} h) = \tilde{V_{2}}^{*n} \tilde{V_{3}}^{*m} V_{1}h.$$ Let $\sum\limits_{n,m} \tilde{V_{2}}^{*n} \tilde{V_{3}}^{*m} V_{1} h_{n,m}$ be a finite sum. Note that 
		\begin{align*}
			\left\| \sum\limits_{n,m} \tilde{V_{2}}^{*n} \tilde{V_{3}}^{*m} V_{1} h_{n,m} \right\|^{2} & = \sum\limits_{n_{1},m_{1}, n_{2}, m_{2}} \la  \tilde{V_{2}}^{*n_{1}} \tilde{V_{3}}^{*m_{1}} V_{1} h_{n_{1},m_{1}},  \tilde{V_{2}}^{*n_{2}} \tilde{V_{3}}^{*m_{2}} V_{1} h_{n_{2},m_{2}} \ra \\
			= &  \sum\limits_{n_{1},m_{1}, n_{2}, m_{2}} \la  \tilde{V_{2}}^{n_{2}} \tilde{V_{3}}^{m_{2}} V_{1} h_{n_{1},m_{1}},  \tilde{V_{2}}^{n_{1}} \tilde{V_{3}}^{m_{1}} V_{1} h_{n_{2},m_{2}} \ra \\
			= & \sum\limits_{n_{1},m_{1}, n_{2}, m_{2}} \la V_{2}^{n_{2}} V_{3}^{m_{2}} V_{1} h_{n_{1},m_{1}}, V_{2}^{n_{1}} V_{3}^{m_{1}} V_{1} h_{n_{2}, m_{2}} \ra \\
			= & \sum\limits_{n_{1},m_{1}, n_{2}, m_{2}} \la V_{1}^{*} V_{1} V_{2}^{n_{2}} V_{3}^{m_{2}} h_{n_{1},m_{1}}, V_{2}^{n_{1}} V_{3}^{m_{1}} h_{n_{2}, m_{2}} \ra \\
			= & \sum\limits_{n_{1},m_{1}, n_{2}, m_{2}} \la V_{2}^{n_{2}} V_{3}^{m_{2}} h_{n_{1},m_{1}}, V_{2}^{n_{1}} V_{3}^{m_{1}} h_{n_{2}, m_{2}} \ra  \\
			& -  \frac{1}{4} \sum\limits_{n_{1},m_{1}, n_{2}, m_{2}} \la V_{2}^{*} V_{2} V_{2}^{n_{2}} V_{3}^{m_{2}} h_{n_{1},m_{1}}, V_{2}^{n_{1}} V_{3}^{m_{1}} h_{n_{2}, m_{2}} \ra \\
			= & \left\| \sum\limits_{n,m} \tilde{V_{2}}^{*n} \tilde{V_{3}}^{*m} h_{n,m} \right\|^{2} -\frac{1}{4} \left\| \tilde{V_{2}} \sum\limits_{n,m} \tilde{V_{2}}^{*n} \tilde{V_{3}}^{*m} h_{n,m} \right\|^{2}.
		\end{align*}
		This proves that $\tilde{V_{1}}$ is well-defined on $\tilde{\cH}$ and satisfies (a), (b) and (c). By Corollary \ref{quasi cor}, the triple $(\tilde{V_{1}}, \tilde{V_{2}}, \tilde{V_{3}})$ is a quasi $\cP-$unitary. Hence, we are done. 
	\end{proof}	
	\section{Pure pentablock isometries}\label{Pure-Isometry}
	It is well known that any pure isometry is unitarily equivalent to the shift operator on the vector-valued Hardy space $H^2\otimes \mathcal E$. We say that a  $\cP-$isometry  $(V_1, V_2, V_3)$ is a pure $\cP-$isometry if $V_3$ is pure. This section will give a model for certain pure $\cP-$isometries. Let us start with a simple example. The pair $(I, 0, M_z)$ on the Hardy space is a pure $\cP-$isometry. The following example gives a non-trivial example of a pure $\cP-$isometry.
	\begin{example}\label{Ex1}
		Let $\left(M_{z_1}, M_{z_2}\right)$ be the pair of multiplication operators on the Hardy space of the bidisc. Then, the triple $\left(\frac{M_{z_1}-M_{z_2}}{2}, M_{z_1}+M_{z_2}, M_{z_1}M_{z_2}\right)$ on $H^2(\mathbb{D}^2)$ is a pure $\cP-$ isometry. 
		
		Indeed, the pair $\left(M_{z_1}+M_{z_2}, M_{z_1}M_{z_2}\right)$ is the symmetrization of the pair $\left(M_{z_1}, M_{z_2}\right)$. Hence, it is a $\Gamma-$isometry. Since the operator $M_{z_1}M_{z_2}$ is a product of two pure isometries, it follows that $M_{z_1}M_{z_2}$ is a pure isometry. Now we shall show that the triple $\left(\frac{M_{z_1}-M_{z_2}}{2}, M_{z_1}+M_{z_2}, M_{z_1}M_{z_2}\right)$ satisfies the following relation: 
		$$\frac{(M_{z_1}-M_{z_2})^*(M_{z_1}-M_{z_2})}{4}= I- \frac{(M_{z_1}+M_{z_2})^*(M_{z_1}+M_{z_2})}{4}.$$
		To that end, consider
		\begin{align*}
			\frac{(M_{z_1}-M_{z_2})^*(M_{z_1}-M_{z_2})}{4}&= \frac{M_{z_1}^*M_{z_1}- M_{z_1}^*M_{z_2}-M_{z_2}^*M_{z_1}+ M_{z_2}^*M_{z_2}}{4}	\\
			&= \frac{2I- M_{z_1}^*M_{z_2}-M_{z_2}^*M_{z_1}}{4}
		\end{align*} and 
		
		\begin{align*}
			I- \frac{(M_{z_1}+M_{z_2})^*(M_{z_1}+M_{z_2})}{4}&= I- \frac{M_{z_1}^*M_{z_1}+M_{z_1}^*M_{z_2}+ M_{z_2}^*M_{z_1}+M_{z_2}^*M_{z_2}}{4}\\
			&=\frac{4I-\left(I+M_{z_1}^*M_{z_2}+ M_{z_2}^*M_{z_1}+I\right)}{4}\\
			&= \frac{2I- M_{z_1}^*M_{z_2}-M_{z_2}^*M_{z_1}}{4}.
		\end{align*}
		Therefore by Theorem \ref{P-isometry} and the definition of pure $\cP-$isometries, the triple $\left(\frac{M_{z_1}-M_{z_2}}{2}, M_{z_1}+M_{z_2}, M_{z_1}M_{z_2}\right)$ is a pure $\cP-$isometry
	\end{example}
	Below, we shall give an example of $\cP-$ isometry, which is not pure. 
	\begin{example}
		Let $\left(M_{z_1}, M_{z_2}\right)$ be the pair of multiplication operators on $H^2(\mathbb{D}^2)$ and $(U_1, U_2)$ be a pair or commuting unitaries on some non-zero Hilbert space $\cH$. Then, by the same calculation as in Example \ref{Ex1}, we can check that the triple
		\begin{equation}\label{Ex22}
			\left(\frac{M_{z_1}\oplus U_1-M_{z_2}\oplus U_2}{2}, M_{z_1}\oplus U_1+M_{z_2}\oplus U_2, M_{z_1}M_{z_2}\oplus U_1U_2\right)
		\end{equation}
		on $H^2(\mathbb{D}^2)\oplus\cH$ is a $\cP-$isomtery . But it is not pure because  $M_{z_1}M_{z_2}\oplus U_1U_2$ is not pure. Hence, the triple \eqref{Ex22} is a $\cP-$isometry, which is not pure $\cP-$isometry.
	\end{example}
	\subsection{Discussion on Pure Pentablock-Isometries}
	Note that if $(V_1, V_2, V_3)$ is a pure $\cP$-isometry, then $(V_2, V_3)$ is a pure $\Gamma-$ isomtery. Therefore, by Theorem 2.4 of \cite{AY-model} (see also \cite{TPS-Adv}), the pair $(V_2, V_3)$ is unitarily equivalent to $(M_{F+zF^*}, M_{z})$ on $H^2(\cD_{V_3^*})$ where $F$ is a numerical contraction and $\cD_{V_3^*}$ is the defect operator. Since $V_1$ commutes with $V_3$, there exists a $\cB(\cD_{V_{3}^{*}})-$valued bounded analytic function $\varphi$ on $\bD$ such that $V_{1}$ is unitaraliy equivalent to $M_{\varphi}$ on $H^2(\cD_{V_3^*})$. Hence $(V_1, V_2, V_3)$ is unitarily equivalent to $(M_{\varphi}, M_{F+zF^*}, M_{z})$ on $H^2(\cD_{V_3^*})$. 
	
	Let $\cE$ be a Hilbert space and $F$ be a numerical contraction on $\cE.$ Then $(M_{F + z F^{*}}, M_{z})$ on $H^{2}(\cE)$ is a pure $\Gamma-$isometry. Since $I - \frac{1}{4}M_{F + z F^{*}}^{*}M_{F + z F^{*}} \geq 0,$ we have 
	$$ I_{\cE} - (F + z F^{*})^{*}(F + zF^{*}) \geq 0 \quad \text{for all $z \in \mathbb{T}$}.  $$
	By operator Fejér-Riesz theorem \cite{Fejer}, there exist two operators $A_{0}, A_{1}$ on $\cE$ such that 
	$$ I_{\cE} - \frac{1}{4}M_{F + z F^{*}}^{*}M_{F + z F^{*}} = (A_{0} + z A_{1})^{*} (A_{0} + z A_{1})  \quad \text{for all $z \in \mathbb{T}$} .$$
	If in addition $A_{0}$ and $A_{1}$ satisfy
	$$(F + z F^{*}) (A_{0} + z A_{1}) = (A_{0} + z A_{1}) (F + zF^{*}) \quad \text{for all $z \in \mathbb{T}$}, $$
	then $(M_{A_{0} + z A_{1}}, M_{F + z F^{*}}, M_{z})$ on $H^{2}(\cE)$ is a pure $\cP-$isometry. To find such $A_{0}$ and $A_{1},$ we need to solve the following operator equations
	\begin{align*}
		A_{0}^{*} A_{0} + A_{1}^{*} A_{1} & = I - \frac{F^{*}F + F F^{*}}{4}\\
		A_{0}^{*}A_{1} & = - \frac{F^{*2}}{4}\\
		F A_{0} &  = A_{0} F \\
		F A_{1} + F^{*} A_{0} & = A_{0}F^{*} + A_{1} F\\
		F^{*} A_{1} &  = A_{1} F^{*}.
	\end{align*}
	In general, we don't know whether the above system of equations is solvable. Note that the operator Fejér-Riesz theorem only guarantees the existence of $A_{0}$ and $A_{1}$, satisfying the first two equations. We can solve these equations and find $A_{0}, A_{1}$ explicitly when $F$ is a normal operator. 
	\begin{thm}
		Let $\cE$ be a Hilbert space and $F$ be a numerical contraction on $\cE.$ If $F$ is a normal operator and $D_{F} = (I - F^{*}F)^{1/2}$, then $ A_{0} = \frac{1}{2}(I + D_{F})$ and $ A_{1} = - \frac{1}{2}(I + D_{F})^{-1} F^{*2}$ satisfy all the five equations above. That is, $(M_{A_{0} + zA_{1}}, M_{F + z F^{*}}, M_{z})$ is a pure $\cP-$isometry. 
	\end{thm}
	\begin{proof}
		Since $F$ is a numerical contraction as well as a normal operator, $F$ is a contraction. So, $D_{F} = (I - F^{*}F)^{1/2}$ makes sense. Let $ A_{0} = \frac{1}{2}(I + D_{F})$ and $ A_{1} = - \frac{1}{2}(I + D_{F})^{-1} F^{*2}.$ Since $F$ is normal, all three operators, $D_{F}, A_{0}$ and $A_{1}$, commute with both $F$ and $F^{*}.$ This proves that $A_{0}$ and $A_{1}$ satisfy the last three equations. The second equation is straightforward. To check the first equation, consider 
		\begin{align*}
			A_{0}^{*}A_{0} + A_{1}^{*}A_{1} & = \frac{1}{4} (I + D_{F})^{2} + \frac{1}{4} F^{2} (I + D_{F})^{-2} F^{*2}\\
			& =  \frac{1}{4} (I + D_{F})^{2} + \frac{1}{4}F^{*2}F^{2} (I + D_{F})^{-2} \\
			& = \frac{1}{4} (I + D_{F})^{2} + \frac{1}{4} (I - D_{F}^{2})^{2} (I + D_{F})^{-2}\\
			& = \frac{1}{4} (I + D_{F}^{2} + 2D_{F}) + \frac{1}{4} (I + D_{F}^{2} - 2D_{F})\\
			& = I - \frac{F^{*}F + F F^{*}}{4}.
		\end{align*}
	\end{proof}

	\section{Beurling-Lax-Halmos Representations of pure $\mathcal P$-isometries}\label{BLH-Rep}
	This section will be focusing on characterizing joint invariant subspaces of pure pentablock isometries. 
	
	Let $(M_{\varphi_1}, M_{\varphi_2}, M_z)$ be a pure $\cP-$isometry on $H^{2} \otimes \cE$. A closed non-zero subspace $\cM$ of $H^2\otimes\cE$ is said to be a joint invariant subspace of $(M_{\varphi_1}, M_{\varphi_2}, M_z)$ if $\cM$ is invariant under $M_{\varphi_1}$, $M_{\varphi_2}$ and $M_z$. We denote by $H^{\infty}(\cE', \cE),$ the space of all holomorphic $\cB(\cE', \cE)-$valued bounded functions on the unit disc. A function $\Theta $ in $H^{\infty}(\cE', \cE)$ is said to be inner if $M_{\Theta}:H^{2} \otimes \cE'\rightarrow H^{2} \otimes \cE$ is an isometry.
	
	\begin{thm}\label{BLH}
		Let $(M_{\varphi_1}, M_{\varphi_2}, M_z)$ be a pure $\cP-$isometry on $H^{2} \otimes \cE$.	Let $\cM$ be a non-zero closed subspace of $H^2\otimes\mathcal{E}$. Then $M$ is a joint invariant subspace of $(M_{\varphi_1}, M_{\varphi_2}, M_z)$ if and only if  there exist a Hilbert space $\mathcal{E}'$ and multipliers $\psi_1, \psi_2$  in $H^\infty(\mathcal{E}', {\mathcal E}')$ and an inner function $\Theta$ in $H^{\infty}({\mathcal E}', {\mathcal E})$ such that $\mathcal{M}= M_{\Theta}(H^{2} \otimes {\mathcal E}')$ and
		$$ M_{\Theta} M_{\psi_j}= M_{\varphi_j} M_{\Theta}, \quad j =1,2$$ 
		with $(M_{\psi_1}, M_{\psi_2}, M_z)$ is a pure $\mathcal{P}$-isometry on $H^2\otimes\mathcal{E}'$.
	\end{thm}
	\begin{proof}
		The converse part is easy to see. We shall prove the necessary part. Since $\cM$ is invariant under $M_z$, there exist a Hilbert space $\mathcal{E}'$ and  an inner multiplier $\Theta$ on $H^\infty (\mathcal{E}', \mathcal{E})$ such that 
		$$\mathcal M=M_{\Theta}\left(H^2\otimes\mathcal{E}'\right).$$
		For $j=1,2$, we also have that 
		$$M_{\varphi_j}M_{\Theta} H^2\otimes\mathcal{E}'\subset M_{\Theta}H^2\otimes\mathcal{E}'$$
		Thus there exist multipliers $\psi_1, \psi_2$ on $H^\infty(\mathcal{E}', \mathcal{E}')$ such that 
		\begin{align}\label{int} 
			M_{\Theta} M_{\psi_1}= M_{\varphi_1} M_{\Theta}\quad \text { and }\quad 	M_{\Theta} M_{\psi_2}= M_{\varphi_2} M_{\Theta}.
		\end{align}
		Now, we need to show that the triple $(M_{\psi_1}, M_{\psi_2}, M_z)$ is a pure $\mathcal{P}$-isometry. By using equation \eqref{int}, we have 
		\begin{align*}
			M_{\psi_1}M_{\psi_2}=M_{\Theta}^*M_{\varphi_1}M_{\Theta}M_{\psi_2}= M_{\Theta}^*M_{\varphi_1}M_{\varphi_2}M_{\Theta} \quad 
			and \\
			M_{\psi_2}M_{\psi_1}= M_{\Theta}^*M_{\varphi_2}M_{\Theta}M_{\psi_2}= M_{\Theta}^*M_{\varphi_2}M_{\varphi_1}M_{\Theta}.
		\end{align*} 
		Since $M_{\varphi_1}$	and $M_{\varphi_2}$ are commuting, it follows that $M_{\psi_1}M_{\psi_2}= M_{\psi_2}M_{\psi_2}$. Also note that $M_z$ commutes with $M_{\psi_j}$, thus $(M_{\psi_1}, M_{\psi_2}, M_z)$ is a commuting triple. Now we shall show that the triple $(M_{\psi_1}, M_{\psi_2}, M_z)$ is a $\mathcal{P}$-isometry. First note that
		\begin{align*}
			M_{\psi_1}^*M_{\psi_1}= \left(M_{\Theta}^*M_{\varphi_1}M_{\Theta}\right)^*M_{\psi_1}=M_{\Theta}^*M_{\varphi_1}^*M_{\Theta}M_{\psi_1}= M_{\Theta}^*M_{\varphi_1}^*M_{\varphi_1}M_{\Theta}.
		\end{align*}
		Since $(M_{\varphi_1}, M_{\varphi_2}, M_z)$ is a $\mathcal{P}$- isometry, by Theorem \ref{P-isometry}, we get 
		\begin{align*}
			M_{\psi_1}^*M_{\psi_1}=M_{\Theta}^*\left(I- \frac{M_{\varphi_2}^*M_{\varphi_2}}{4}\right)M_{\Theta}.
		\end{align*}
		Again, a simple calculation with the help of equation \eqref{int} will give that
		\begin{align*}
			M_{\psi_1}^*M_{\psi_1}=I- \frac{M_{\psi_2}^*M_{\psi_2}}{4}.	
		\end{align*}
		Since $M_z$ is pure, therefore, by Theorem \ref{P-isometry}, the triple $(M_{\psi_1}, M_{\psi_2}, M_z)$ is a pure $\mathcal{P}$-isometry.	
	\end{proof}	
	\begin{center}
		\textbf{Statements and Declarations}
	\end{center}

	\textbf{Funding:} The research works of the first author are supported by the Prime Minister Research Fellowship PM/MHRD-20-15227.03. This work was done when the second author was a PhD student at the Indian Institute of Science, Bangalore. Currently, the second author's research is partially supported by a PIMS postdoctoral fellowship.
	
	\vspace{0.1in} \noindent\textbf{Acknowledgement:}
	The authors thank Professor Tirthankar Bhattacharyya and Dr. Haripada Sau for their valuable discussions and suggestions. The second author also thanks Prof. Raul E. Curto for some discussion.
	
	The authors thank the referee for their helpful comments and suggestions that significantly improved the paper.

\end{document}